\documentclass[11pt
]{amsart}
\usepackage{amsmath, amssymb} 

\usepackage{hyperref}

	\normalbaselines						  %

\numberwithin{equation}{section}

\newtheorem{thm}{Theorem}[section]

\newtheorem{lm}[thm]{Lemma}

\theoremstyle{remark}

\newtheorem*{rem*}{Remark}

\theoremstyle{definition}

\newtheorem*{df*}{Definition}

\newenvironment{entry}
{\begin{list}{X}%
  {%
      \setlength{\labelwidth}{55pt}%
      \setlength{\leftmargin}{\labelwidth}
      \addtolength{\leftmargin}{\labelsep}%
      \setlength{\itemsep}{.4pc}
   }%
}%
{\end{list}}

\newcounter{vremennyj}

\newcommand{\pbar}{\ensuremath{\bar{\partial}}}

\newcommand{\la}{\lambda}

\newcommand{\be}{\mathbf{e}}
\newcommand{\bx}{\mathbf{x}}

\newcommand{\T}{\mathbb{T}}

\newcommand{\f}{\varphi}
\newcommand{\vf}{\varphi}

\newcommand{\D}{\mathbb{D}}
\newcommand{\C}{\mathbb{C}}

\newcommand{\p}{\partial}

\newcommand{\nm}{\,\rule[-.6ex]{.13em}{2.3ex}\,}

\newcommand{\1}{\mathbf{1}}

\newcommand{\cX}{\mathcal{X}}

\newcommand{\bP}{\mathbb{P}}

\newcommand{\bK}{\mathbf{K}}

\newcommand{\fdot}{\,\cdot\,}

\newcommand{\ci}[1]{_{ {}_{\scriptstyle #1}}}

\begin{document}

\title
{A remark on the reproducing kernel thesis for Hankel operators}

\author{Sergei Treil}
\thanks{
This material is based on the work supported by the National Science Foundation under the grant  DMS-0800876. Any opinions, findings and conclusions or recommendations expressed in this material are those of the author and do not necessarily reflect the views of the National Science Foundation. }
\address{Department of Mathematics, Brown University, 151 Thayer 
Str./Box 1917,      
 Providence, RI  02912, USA }
\email{treil@math.brown.edu}
\urladdr{http://www.math.brown.edu/\~{}treil}

\date{}

\begin{abstract} 
In this paper we give a simple proof of the so called reproducing kernel thesis for Hankel operators. 

\end{abstract}

\maketitle

\setcounter{tocdepth}{1}
\tableofcontents

\section*{Notation}

\begin{entry}
\item[$:=$] equal by definition;
\item[$\C$] the complex plane;
\item[$\D$] the unit disk, $\D:=\{z\in\C: |z|<1\}$;
\item[$\T$] the unit circle, $\T:=\p\D=\{z\in\C:|z|=1\}$;
\item[$\widehat f(n)$] Fourier coefficient of the function $f$, $\widehat f(n) := (2\pi)^{-1} \int_\T f(z) z^{-n} |dz|$
\item[$L^p=L^p(\T)$]  Lebesgue spaces with respect to the normalized Lebesgue measure $(2\pi)^{-1} |dz| $ on $\T$;
\item[$H^p$] Hardy spaces, $H^p:= \{ f\in L^p(\T): \widehat f(n)=0 \ \forall n<0\}$;
\item[$H^2_-$]  $H^2_- := L^2(\T)\ominus H^2 = \{ f\in L^2(\T): \widehat f(n)=0 \ \forall n \ge 0\}$;
\item[$H^p(E)$] vector-valued Hardy spaces with values in a separable Hilbert space $E$;
\item[$H^2_-(E)$] vector-valued $H^2_-$;
\item[$\bP_+$, $\bP_-$] orthogonal projections onto $H^2$ and $H^2_-$ respectively;  
\item[$\|\cdot\|, \nm\cdot \nm$]  norm;  when dealing with vector-valued functions we will use the symbol $\|\cdot\|$ (usually with a
subscript) for the norm in a functions space, while $\nm\cdot\nm$ is used for the
norm in the underlying vector  space.
 Thus for a vector-valued function
$f$ the symbol
$\|f\|_2$ denotes its
$L^2$-norm, but the symbol $\nm f\nm$ stands for the scalar-valued function whose
value at a point $z$ is the norm of the vector  $f(z)$.
 
\end{entry}

\section{Introduction and main results}

A Hankel operator is a bounded linear operator $\Gamma : H^2 \to H^2_-$ such that its matrix with respect to the standard bases $\{ z^n\}_{n\ge 0}$ and $\{ \overline z^{n+1} \}_{n\ge0}$ in $H^2$ and $H^2_-$ respectively, depends on the sum of indices, i.e.~has the form
\[
\{\gamma_{j+k+1}\}_{j, k=0}^\infty. 
\]
If one defines 
\[
\f_- := \Gamma \1 = \sum_{k=1}^\infty \gamma_k \overline z^{k}, \qquad z\in \T, 
\]
then the action of $\Gamma$ on polynomials $f$ is given by 
\begin{equation}
\label{hank1}
\Gamma g = \bP_- (\f_- f). 
\end{equation}
The function $\f_-$ is called the antianalytic symbol of the operator $\Gamma$. 

In this paper we will also be dealing with the vectorial Hankel operators $\Gamma: H^2 \to H^2_-(E)$, where $E$ is an auxilary (separable) Hilbert space. In this case the entries $\gamma_k$ are operators $\gamma_k:\C\to E$ and are naturally identified with vectors in $E$. The symbol $\f_-$ then is a vector-valued function in $H^2_-(E)$. 

Note that  in \eqref{hank1} we can replace $\f_-$ by $\f\in L^2(\T\to E)$ such that $\f - \f_-\in H^2(E)$  (so $\widehat \f(n) = \gamma_n $ for all $n<0$). Such functions $\f$ is called   \emph{a symbol} of the operator $\Gamma$. Unlike the antianalytic symbol $\f_-$, the symbol $\f$ is not unique. Note also that for any symbol $\f$ of the Hankel operator $\Gamma$ the estimate $\|\Gamma\|\le \|\f\|_\infty$ holds, and the famous Nehari Theorem states that one can find a symbol $\f$ such that $\|\Gamma\|=\|\f\|_\infty$

In this paper we deal with the so-called \emph{(pre)Hankel} operators (non-standard term), which are not assumed to be bounded, but only defined on polynomials (and have the Hankel matrix $\{\gamma_{j+k+1}\}_{j, k=0}^\infty$). In this case the anianalytic symbol $\f_-$ is also in $H^2_-$, and the action of $\Gamma$ on polynomials is still given by \eqref{hank1}. Using uniform approximation by polynomials, we can easily see that a (pre)Hankel operator $\Gamma$ can be defined on $H^2\cap C(\T)$ and  that its action on $H^2\cap C(\T)$ is still given by \eqref{hank1}.

Let us recall that the normalized reproducing kernel $k_\la$, $\la\in \D$ of $H^2$ is given by
\begin{align}
\label{eq.nrk}
k_\la(z) := \frac{(1-|\la|^2)^{1/2}}{1-\overline \la z},  
\end{align}
and that $\|k_\la\|_2 =1$.

The goal of this paper is to give an elementary proof of the following well known 
result. 

\begin{thm}[Reproducing kernel thesis for Hankel operators]
\label{t.rkt}
Let $\Gamma$ be a possibly vectorial (pre)Hankel operator such that 
\begin{align*}
\sup_{\la\in \D} \| \Gamma k_\lambda \|_2 \le A < \infty
\end{align*}
Then $\Gamma$ is bounded and $\|\Gamma\| \le 2\sqrt{e} A$. 
\end{thm}

This theorem for the  scalar-valued case (with some constant $C$ instead of $2\sqrt{e}$) was  published in \cite{Bonsall1984}, and is widely used in theory of Hankel operators. 

The proof presented in this note is quite elementary and uses only Green's formula: the standard proof uses Nehari Theorem, $H^1$-BMO duality and the fact that the so-called  Garsia norm is an equivalent norm in BMO. 

While Nehari Theorem is a basic fact in the theory of Hankel operators, and  the other facts   are standard and well-known results in Harmonic analysis, it is still interesting to know that non of these results is needed for the prof of the Reproducing Kernel Thesis for Hankel operators (Theorem \ref{t.rkt}).

Finally, let us emphasize, that while the target space of our operator is a vector-valued space $H^2_-(E)$, the domain is usual scalar-valued $H^2$ (more precisely, initially a dense subset of $H^2$). It is known that the reproducing kernel thesis fails for operator-valued Hankel operators: while it is true for Hankel operators  acting from $H^2(\C^d)$ to $H^2_-(E)$, the constant grows logaritmically in $d$, see for example \cite{NazTreVolPis-2002}.

\section{Proof of the main result}

Let us fix some notation. For $f\in L^1(\T)$ and $z\in \D$ let $f(z) $ denote the Poison (harmonic) extension of $f$ at the point $z$. Thus, for $\f\in L^2(\T\to E)$ the symbol $\nm \f(z)\nm^2$ is the square of tne norm (in $E$) of the harmonic extension of $\f$ at the point $z\in\D$, and $\nm\f\nm^2(z)$ is the harmonic extension of $\nm\f\nm^2\Bigm|_{\T}$ at $z$.

\subsection{Hankel operators and reproducing kernels}

Let us recall that the reproducing kernel $K_\la$, $\la\in \D$ of the Hardy space $H^2$ is given by 
\[
K_\la(z) = \frac{1}{1-\overline\la z}. 
\]
It is called the \emph{reproducing kernel} because for all $f\in H^2$
\begin{align}
\label{rkprop}
(f, K_\la) = f(\la). 
\end{align}
Note, that because for each $\la\in\D$ the function $K_\la$ is bounded, the simple approximation argument implies that the reproducing kernel identity \eqref{rkprop} holds for all $f\in H^1$.

Using the reproducing kernel property \eqref{rkprop} with $f = K_\la$ on gets
\[
\|K_\la\|_2^2 = (K_\la, K_\la) = (1-|\la|^2)^{-1}, 
\]
so the normalized reproducing kernel $k_\la := \|K_\la\|_2^{-1} K_\la$ is given by \eqref{eq.nrk}.

The following lemma is well known, it can be found, for example (in implicit form)  in  \cite{Bonsall1984}. We present it here only for the convenience of the reader. 

\begin{lm}
\label{l.GK_la}
Let $\Gamma$ be a (pre)Hankel operator, and let $\f\in H^2_-(E)$ be its antianlytic symbol $\f=\sum_{k=1}^\infty \gamma_k \overline z^{k}$ (to simplify the notation we skip subscript ``$-$'' and use $\f$ instead of $\f_-$). Then for all $\lambda \in \D$
\begin{align*}
\| \Gamma k_\la \|_2^2 = \nm \f \nm^2(\la) - \nm\f(\la)\nm^2. 
\end{align*}
\end{lm}

To proof the lemma we will need the following well-known fact. 

\begin{lm}
\label{l.toepl-eigenvalue}
Let $\f\in H^2_-(E)$. Then for all $\la\in \D$ 
\[
\bP_+ (\f k_\la) =  k_\la  {\f(\la)}  .
\]
\end{lm}
\begin{proof}
Let us first proof this lemma for scalar-valued $\f\in H^2_-$. 

Let $f:= \bP_+ (\f K_\la)$, where $K_\la$ is the reproducing kernel for $H^2$. Any $f\in H^2$ can be decomposed as 
\begin{align*}
f= c K_\la + f_0, 
\end{align*}
where $f_0(\la) = 0$, and $c= (1-|\la|^2) f(\la)$; note that $K_\la\perp f_0$.  

Let us first show that $f_0=0$ for $f=\bP_+ (\f K_\la)$. Notice that $\overline \f f_0\in H^1$ because  $\overline \f, f_0\in H^2$, so we can get using the reproducing kernel property \eqref{rkprop}
\begin{align*}
\|f_0\|_2^2 = (f_0, \bP_+ (\f K_\la)) = (f_0, \f K_\la) = (\overline \f f_0, K_\la) = (\overline \f f_0)(\la) = \overline{ \f(\la)} f_0(\la) = 0. 
\end{align*}

On the other hand, 
\begin{align*}
(K_\la , f) = (K_\la , \bP_+ (\f K_\la)) = (K_\la , \f K_\la) = (\overline \f K_\la, K_\la) & = \overline{\f(\la)} K_\la(\la)\\
&  = 
\overline{\f(\la)} (1-|\la|^2)^{-1}. 
\end{align*}
Therefore
\begin{align*}
(\bP_+(\f K_\la), K_\la) = \f(\la) (1-|\la|^2)^{-1} = \f(\la) \|K_\la\|_2^2, 
\end{align*}
so 
\[
\bP_+(\f K_\la) = \f(\la) K_\la. 
\]
Multiplying this identity by $(1-|\la|^2)^{1/2}$ we get the conclusion of the lemma for scalar-valued $\f$. 

The general vector-valued case can be easily obtained from the scalar-valued case by fixing an orthonormal basis $\{\be_k\}_k$ and applying the scalar-valued result to coordinate functions $\f_k$, $\f_k (z) = (\f(z), \be_k)\ci E$. 
\end{proof}

 \begin{proof}[Proof of Lemma \ref{l.GK_la}]
 Function $\f k_\la$ can be decomposed as the orthogonal sum
\begin{align*}
\f k_\la = \bP_+ (\f k_\la) + \bP_- (\f k_\la), 
\end{align*}
so 
\begin{align*}
\|\Gamma k_\la\|_2^2 = \| \bP_- (\f k_\la) \|_2^2  = \| \f k_\la \|_2^2 - \| \bP_+ (\f k_\la) \|_2^2. 
\end{align*}

Noticing that
\[
|k_\la(z)|^2 = \frac{1-|\la|^2}{|1-\overline \la z|^2}
\]
we can write 
\begin{align*}
\| \f k_\la \|_2^2 = \frac1{2\pi} \int_\T \nm\f(z) \nm^2 |k_\la(z)|^2 |dz| = \nm\f\nm^2(\la). 
\end{align*}
According to Lemma \ref{l.toepl-eigenvalue} $\bP_+ (\f k_\la) = \f(\la) k_\la$, so $\| \bP_+ (\f k_\la) \|_2^2 = \nm\f(\la)\nm^2$.
\end{proof}

\subsection{Green's formula and Littlewood--Paley identities}
We need several well-known facts. 

The first one is the standard Green's formula for the unit disc.
\begin{lm}
\label{l.Green}
Let $f\in C^2(\D) \cap C(\overline\D)$. Then 
\[
\frac1{2\pi}  \int_\T U(z) |dz| - U(0) = \frac1{2\pi} \int_\D\Delta U(z) \ln\frac1{|z|} dA(z) 
\]
\end{lm}

Applying this lemma to $U(z) = \nm f(z)\nm^2$, $f\in H^2(E)$ and noticing that $\Delta U = 4\p\pbar U = 4 \nm f'\nm^2$ we get the following Littlewood--Paley identity. 

\begin{lm}
\label{l.L-P-1}
Let $f\in H^2(E)$. Then 
\[
\|f\|_2^2 = \frac{2}{\pi} \int_\D \nm f'(z)\nm^2 \ln\frac1{|z|} dA(z) + \nm f(0)\nm^2 .
\]
\end{lm}
 Of course we have to first apply Lemma \ref{l.Green} to $\nm f(rz)\nm^2$, $r<1$ and then take limit as $r\to 1$. 

The following lemma is also well-known, see for example Lemma 6 in Appendix 3 of the monograph \cite{Nik-shift}
\begin{lm}
\label{l.Uchi}
Let $u$ be a $C^2$ subharmonic function ($\Delta u\ge 0$) in the unit disc $\D$, and let $0\le u(z) \le 1$ for all $z\in \D$. Then for all $f\in H^2(E)$
\[
\frac1{2\pi}\int_\D \Delta u(z)  \nm f(z)\nm^2 \ln\frac1{|z|} dA(z) \le e \|f\|_2^2
\]
\end{lm}

\begin{proof}
Replacing $u$ and $f$ by $u(rz)$ anf $f(rz)$, $r<1$ and then taking limit as $r\to 1$ we can always assume without loss of generality that $u$ and $f$ are continuous up to the boundary of $\D$, so the Green's formula (Lemma \ref{l.Green}) applies to 
 $U(z) = e^{u(z)} \nm f(z)\nm^2$. Direct computation using the fact that $\Delta =4\p\pbar$  shows that 
\[
\Delta \bigl( e^{u (z)} \nm f(z)\nm^2\bigr) = e^u (\Delta u) \nm f \nm^2 + 4e^u \nm(\p u) f + \p f \nm^2 \ge  (\Delta u) \nm f\nm^2. 
\]
Then denoting $d\mu(z) =(2\pi)^{-1}\ln|z|^{-1} dA(z)$, we can write using  Green's formula (Lemma \ref{l.Green}):  
\begin{align*}
\frac1{2\pi}\int_\D \Delta u \, \nm f \nm^2\, d\mu & \le \int_\D \Delta \bigl( e^\vf  \nm f\nm^2\bigr)\, d\mu 
\\ &= 
\frac{1}{2\pi}\int_\T e^\vf  \nm f\nm^2 \, |dz| - e^{\vf(0)} \nm f(0)\nm^2 \le e \frac{1}{2\pi}\int_\T   \nm f\nm^2 \, |dz| =e \, \|f\|_2^2.
\end{align*}
\end{proof}

\subsection{Proof of Theorem \ref{t.rkt}} By homogeneity it is sufficient to prove theorem only for $A=1$, so let us assume that 
\[
\sup_{\la\in\D}\|\Gamma k_\la \|_2\le 1. 
\]

Let us introduce some notation. Fix an orthonormal basis $\{\be_k\}_k$ in $E$ and for a vector $\bx =\sum_k x_k\be_k \in E$ (of course $x_k = (\bx, e_k)\ci E$) let $\overline \bx$ be the ``complex conjugate'' vector, $\overline\bx := \sum_k \overline x_k\be_k$. So for the function $h$ with values in $E$ the symbol $\overline h$ denotes the function obtained by taking complex conjugates of the coordinate functions of $h$ (the orthonormal basis $\{e_k\}_k$ is assumed to be fixed). 

Let $\f$ be the antianalytic symbol of the Hankel operator $\Gamma$, so $\Gamma=\Gamma_\f$. Recall that for $z\in\D$ we use $\f(z)$ to denote the harmonic extension of $\f$ to the unit disc, so $\overline \f\in H^2(E)$. 

It is sufficient to estimate the operators $\Gamma_{\f_r}$, $\f_r(z):= \f(rz)$, $r\in(0,1)$, so without loss of generality we can assume that $\overline \f$ is analytic in some bigger than $\D$ disc. 

We want to estimate 
\[
(\Gamma f , \overline g) = \frac{1}{2\pi} \int_\T (\f f ,\overline g )\ci E |dz|, \qquad f\in H^2, \ \overline g \in H^2_-(E) \ (\text{equiv. } g\in z H^2(E)). 
\]
It is sufficient to check the boundedness on a dense set, so we can assume that $f$ and $g$ are polynomials, so we can apply the Green's formula. Since $f, g $ and $\overline \f$ are analytic in $\D$ and $\Delta = 4\p\pbar$, we get
$\pbar(\f f, g)\ci E = (f (\pbar\f) , g)\ci E$ and
\begin{align*}
\Delta (\f f, g)\ci E  & = 4 \left( (f(\pbar \f),\pbar \overline g)\ci E + (f'(\pbar \f), \overline g)\ci E \right) 
\\
& = 4 \left( (f(\pbar \f), \overline{ g'})\ci E + (f'(\pbar \f), \overline g)\ci E \right) . 
\end{align*}
Therefore, using the Green's formula (Lemma \ref{l.Green}) and the fact that the function $(\f f, \overline g)\ci E$ vanishes at the origin, we get 
\begin{align*}
(\Gamma f ,\overline g) = \frac{1}{2\pi} \int_\T (\f f, \overline g)\ci E |dz| =\frac{2}{\pi} \int_\D  \left( (f (\pbar \f),\overline{g'})\ci E + (f' (\pbar \f),\overline g)\ci E \right) \ln\frac1{|z|} dA(z) .
\end{align*}
We can estimate by Cauchy--Schwartz
\begin{align*}
&\left| \frac{2}{\pi} \int_\D  (f (\pbar \f), \overline{g'})\ci E  \ln\frac1{|z|} dA(z)  \right| 
\\
& \qquad \qquad \quad
\le \left(\frac{2}{\pi}\int_\D \nm(\pbar \f)  \nm^2 |f|^2 \ln\frac1{|z|} dA(z)\right)^{1/2}  \left(\frac{2}{\pi}\int_\D \nm g' \nm^2 \ln\frac1{|z|} dA(z)\right)^{1/2} .
\end{align*}
By Lemma \ref{l.L-P-1} 
\begin{equation}
\label{int1}
\frac{2}{\pi}\int_\D \nm g' \nm^2 \ln\frac1{|z|} dA(z) \le \|g\|_2^2. 
\end{equation}
To estimate the first integral, let us define $u(z) = 1 +   \nm \f(z)\nm^2 -\nm \f\nm^2(z)$, and notice that $\Delta u = 4 \nm \pbar \f\nm^2$. It follows from the assumption $\sup_{\la\in\D}\|\Gamma k_\la\|_2\le 1$ and Lemma \ref{l.GK_la} that 
\[
0\le u(z) \le 1\qquad \forall z\in \D,  
\]
so by Lemma \ref{l.Uchi}
\begin{equation}
\label{int2}
\frac{2}{\pi}\int_\D \nm (\pbar \f)\nm^2 | f |^2 \ln\frac1{|z|} dA(z) \le e \|f\|_2^2. 
\end{equation}
Gathering together etimates \eqref{int1} and \eqref{int2} we get
\[
\left| \frac{2}{\pi} \int_\D \left( f(\pbar \f) , \overline{g'} \right)\ci E  \ln\frac1{|z|} dA(z)\right| \le \sqrt{e} \|f\|_2\|g\|_2.
\]

Similarly, 
\begin{align*}
&\left| \frac{2}{\pi} \int_\D  (f' (\pbar \f), \overline{g})\ci E  \ln\frac1{|z|} dA(z)  \right| 
\\
& \qquad \qquad \quad
\le \left(\frac{2}{\pi}\int_\D \nm(\pbar \f)  \nm^2 \nm g\nm^2 \ln\frac1{|z|} dA(z)\right)^{1/2}  \left(\frac{2}{\pi}\int_\D | f' |^2 \ln\frac1{|z|} dA(z)\right)^{1/2},
\end{align*}
so interchanging $f$ and $g$ in the above reasoning  we get the estimate 
\[
\left| \frac{2}{\pi} \int_\D \left(  f' (\pbar \f), g \right)\ci E \ln\frac1{|z|} dA(z)\right| \le \sqrt{e} \|f\|_2\|g\|_2, 
\]
so $|(\Gamma f, g )| \le 2\sqrt{e} \|f\|_2\|g\|_2$. \hfill\qed

\section{Concluding remarks}

The main idea of using only Green's formula (and Lemma \ref{l.Uchi}) 
goes back to \cite{Pet-Tre-Vick2007}, where  the reproducing kernel thesis for the Carleson  embeding theorem for the disc and for the unit ball in $\C^n$ was proved using similar technique; for the disc the estimate $\sqrt{2e}$ for the norm of the embedding operator\footnote{Compare with $2\sqrt{e}$ for Hankel operators}  was obtained, see Theorem 0.2 there.  

However the proof in the present paper is much simpler than in \cite{Pet-Tre-Vick2007}. Namely, the proof in \cite{Pet-Tre-Vick2007} required some not completely trivial comutations and estimates; in the present paper all the computations (modulo known facts such as Lemmas  \ref{l.GK_la}--\ref{l.Uchi}) can be done in one's head. 

Using the above mentioned estimate for the Carleson mebedding theorem  from \cite{Pet-Tre-Vick2007}, B.~Jacob, J..~Partington, and S.~Pott obtained in \cite{Jac-Pott-Part-JFA_2010} the estimate $4\sqrt{2e}$   for the reproducing kernel thesis for Hankel operators. Their proof also used Green's formula, but the proof presented here besides giving a better constant is significantly simpler and much more streamlined (in particular, because it does not use the result from \cite{Pet-Tre-Vick2007}). 

Also, Theorem \ref{t.rkt} here can be use to give an explicit constant in the reproducing kernel thesis for the so-called generalized embedding theorem, described below in Section \ref{s.GET}, and in particular for the Carleson emebedding theorem, although for the Carleson embedding theorem it gives worse constant than one obtained in \cite{Pet-Tre-Vick2007}. 


\subsection{Generalized embedding theorem}
\label{s.GET}
Let  $\theta\in H^\infty$ be an inner function, and let $\bK_\theta$ be   the corresponding backward whift invariant subsapce
\[
\bK_\theta := H^2 \ominus \theta H^2. 
\]
It is well known (see for example Projection Lemma in \cite[p.~34]{Nik-shift}) and is easy to prove that the orthogonal projection $P_\theta$ from $H^2$ onto $\bK_\theta$ is given on the unit circle $\T$  by 
\begin{equation}
\label{P_theta}
P_\theta f =  f - \theta \bP_+(\overline \theta f)  = \theta \bP_-(\overline \theta f), \qquad f\in H^2. 
\end{equation}
Let $(\cX, \mu)$ be a measure space, and $\theta_\lambda$, $\la\in\cX$ be a measurable family of inner functions (meaning that the function $(z, \la)\mapsto \theta_\la(z)$ is  measurable). The equality \eqref{P_theta} implies that the projection-valued function $\la \mapsto P_{\theta_\la}$ is measurable (in weak, and so in strong sense), so one can ask on what conditions on the measure $\mu$ the following \emph{generalized embedding theorem}
\begin{equation}
\label{GET}
\int_\cX \| \bP_{\theta_\la} f \|_{H^2}^2 d\mu(\la) \le C \|f\|_{H^2}^2 \qquad \forall f\in H^2
\end{equation}
holds. 

Note that if $\theta$ is an elementary Blaschke factor, 
\[
\theta(z) =\frac{z-\la}{1-\overline\la z} , 
\]
then the corresponding space $\bK_\theta$ is spanned by the reproducing kernel $k_\la$, and 
\[
P_\theta f = (f, k_\la)k_\la = (1-|\la|^2)^{1/2} f(\la) k_\la, 
\]
so $\|P_\theta f\|_2^2 = (1-|\la|^2) |f(\la)|^2$. 

Therefore for $\cX=\D$ and $\theta_\la(z) = (z-\la)/(1-\overline\la z)$, $\la\in\D$, the estimate \eqref{GET} reduces to the classical Carleson embedding theorem, and \eqref{GET} holds if and only if the measure $(1-|\la|^2 )d\mu(\la)$ is Carleson. 

Define a Hankel operator $\Gamma: H^2\to H^2_-(L^2(\mu))$ by 
\[
\Gamma f (z, \la) = \Gamma_{\overline{\theta_\la}} f (z) = \bP_-(\overline{\theta_\la} f)(z), \qquad f\in H^2, \quad z\in\T, \ \la\in\D. 
\]
It follows from \eqref{P_theta} that  
\[
\| \Gamma f (\fdot, \la)\|_2 = \|P_{\theta_\la} f\|_2, 
\]
so \eqref{GET} is equivalent to the estimate $\|\Gamma\|\le \sqrt{C}$.

But for Hankel operators the reproducing kernel thesis holds, and Theorem \ref{t.rkt} implies that if 
\[
\int_\D \|P_{\theta_\la} k_a\|_2^2 d\mu(\la) \le A\|f\|_2^2 \qquad \forall a\in \D, 
\]
then \eqref{GET} holds with $C=4eA$. 

The fact that the reproducing kernel thesis holds for the generalized embedding theorem (with some constant) was proved in \cite{Treil_Embed_AA-1989}; the above reasoning connecting \eqref{GET} and the boundedness of the Hankel operator $\Gamma$ is essentially taken from there. Theorem \ref{t.rkt} from the present paper just gives us an explicit constant. 

It also gives a simpler proof of Theorem 0.2 from \cite{Pet-Tre-Vick2007} (reproducing kernel thesis for the Carleson emebedding theorem) but with a worse constant ($4e$ vs $2e$ in \cite{Pet-Tre-Vick2007}\footnote{if one considers the  norms of the embedding operators one should take square roots of the above constants}).

\def\cprime{$'$}
\providecommand{\bysame}{\leavevmode\hbox to3em{\hrulefill}\thinspace}


\begin{thebibliography}{10}

\bibitem{Bonsall1984}
F.~F.~Bonsall, \emph{Boundedness of {H}ankel matrices}, J. London Math. Soc.
  (2) \textbf{29} (1984), no.~2, 289--300.  
  
\bibitem{Jac-Pott-Part-JFA_2010}
B.~Jacob, J..~Partington, and S.~Pott, \emph{Weighted
  interpolation in {P}aley-{W}iener spaces and finite-time controllability}, J.
  Funct. Anal. \textbf{259} (2010), no.~9, 2424--2436.


\bibitem{NazTreVolPis-2002}
F.~Nazarov, G.~Pisier, S.~Treil, and A.~Volberg, \emph{Sharp estimates in
  vector {C}arleson imbedding theorem and for vector paraproducts}, J. Reine
  Angew. Math. \textbf{542} (2002), 147--171.
  
  

\bibitem{Nik-shift}
N.~K.~Nikol{\cprime}ski\u{\i}, \emph{Treatise on the shift operator},
  Grundlehren der Mathematischen Wissenschaften [Fundamental Principles of
  Mathematical Sciences], vol. 273, Springer-Verlag, Berlin, 1986, Spectral
  function theory, With an appendix by S. V. Hru\v s\v cev [S. V. Khrushch\"ev]
  and V. V. Peller, Translated from the Russian by Jaak Peetre.  
  
  
  
\bibitem{Pet-Tre-Vick2007}
S.~Petermichl, S.~Treil, and B.~Wick, \emph{Carleson potentials
  and the reproducing kernel thesis for embedding theorems}, Illinois J. Math.
  \textbf{51} (2007), no.~4, 1249--1263.
  
  
\bibitem{Treil_Embed_AA-1989}
S.~Treil, \emph{Hankel operators, embedding theorems and bases of co-invariant
  subspaces of the multiple shift operator}, Algebra i Analiz \textbf{1}
  (1989), no.~6, 200--234. 


\end{thebibliography}
\end{document}